\documentclass[review]{elsarticle}

\usepackage{lineno,hyperref}
\usepackage{graphicx}
\usepackage{amssymb}
\usepackage{epstopdf}
\usepackage{epsfig}
\usepackage{amsmath}
\usepackage{amsthm}
\usepackage{caption}
\newtheorem{definition}{Definition}
\newtheorem{theorem}{Theorem}
\newtheorem{lemma}{Lemma}

\makeatletter
\renewcommand{\maketag@@@}[1]{\hbox{\m@th\normalsize\normalfont#1}}%
\makeatother
\modulolinenumbers[5]
\journal{Chaos, Solitons and Fractals}

\begin{document}
\captionsetup[figure]{labelfont={bf},labelformat={default},labelsep=period,name={Fig.}}
\begin{frontmatter}

\title{An explicit Hopf bifurcation criterion of fractional-order systems with order $1<\alpha <2$}

\author{Jing Yang}
\author{Xiaoxue Li$^*$,Xiaorong Hou$^*$}
\address{University of Electronic Science and Technology of China, School of Automation Engineering, Chengdu 611731,
China}
\cortext[mycorrespondingauthor]{Corresponding author}
\ead{lixx@uestc.edu.cn,houxr@uestc.edu.cn}

%\author{Yajun Li}
%\address{State Grid Xinyang Electric Power Supply Company, Xinyang,
%China}


\begin{abstract}
A Hopf bifurcation criterion of fractional-order systems with order $1<\alpha <2$ is established in this paper, in which all conditions are explicitly expressed by parameters without solving the roots of the relevant characteristic polynomial of Hopf bifurcation conditions. It avoids the problem that existing methods may fail due to the computational complexity  in the multi-parameter situation. The bifurcation hyper-surface of multi-parameter can be obtained directly.
%A fractional-order Routh-Hurwitz matrix, which corresponds to the characteristic polynomial of Jacobian matrix, is constructed to describe the relationship between coefficients of the characteristic polynomial and roots distribution.
%The Hopf bifurcation results of multi-parameter systems or high-dimensional systems can be obtained directly.
%This criterion is particularly efficient to analyze Hopf bifurcation of fractional-order systems with multiple parameters.
\end{abstract}

\begin{keyword}
Hopf bifurcation\sep Fractional-order system\sep Multi-parameter\sep Generalized Routh-Hurwitz criterion
%\sep Incommensurate orders\sep Stability \sep Uncertain system\sep Parameter space
%\texttt{elsarticle.cls}\sep \LaTeX\sep Elsevier \sep template
%\MSC[2010] 00-01\sep  99-00
\end{keyword}

\end{frontmatter}

\linenumbers

\section{Introduction}
Hopf bifurcation of fractional-order dynamical systems was discussed numerically \cite{el2009stability,li2014hopf,vcermak2019stability,deshpande2017hopf}.
%One of the fractional-order Hopf bifurcation conditions called the eigenvalues condition is that Jacobian matrix has a pair conjugate complex eigenvalues on the critical line, while the other eigenvalues are stable \cite{ma2016complexity}.
These existing methods are valid to obtain the results of Hopf bifurcation with the idea of analyzing eigenvalues of the Jacobian matrix of low-dimensional systems.
%need to first solve the eigenvalues of the Jacobian matrix, then test whether the imaginary part and real part of the eigenvalues satisfy $\frac{Im(\lambda(\mu^{*}))}{Re(\lambda(\mu^{*}))}=tan\left ( \frac{\alpha \pi }{2} \right )$, where $\mu^{*}$ is the critical value of the bifurcation parameter, $\alpha$ is the fractional order.
To ensure the feasibility of methods in high-dimensional situation, some researchers simplified multi-parameter to a single parameter. For example, in the problem of Hopf bifurcation about fractional-order neural network systems, several neuron parameters are set to a same bifurcation parameter or the sum of some parameters is set to a single bifurcation parameter \cite{huang2017dynamical,2018Effects}, although the fact is that different neuron parameters have different values \cite{wei2004bifurcation,yan2006hopf}. For high-dimensional fractional-order systems with multi-parameter, an online method \cite{YANG2022111714} is suitable to determine the Hopf bifurcation hyper-surface through a visual representation of the parameter space if the systems have fewer parameters. However, the online method may fail because of the computational complexity of multi-parameter.
\par
In this paper,
%a fractional-order Routh-Hurwitz matrix, which corresponds to the characteristic polynomial of Jacobian matrix, is constructed to describe the relationship between coefficients of the characteristic polynomial and roots distribution.
%Based on the generalized Routh-Hurwitz criterion \cite{gantmakher2000theory} and the subresultant theorem \cite{ABDELJAOUED2009588},
we establish a criterion for analyzing the Hopf bifurcation of fractional-order systems, where the expressions are all explicit ones of multiple parameters. This criterion can give results directly, which avoids the problem that the existing methods may fail due to the computational complexity in the multi-parameter situation. The Caputo fractional derivative is employed in this paper \cite{1999Fractional}.
\par
%\par
%\begin{equation}
%\frac{d^{q }f(t)}{dt^{q }}=\frac{1}{\Gamma (n-q)}\int_{0}^{\infty}\frac{f^{(n)}(\tau )}{(t-\tau )^{q+1-n}}d\tau,
%\end{equation}
%where $n-1< q< n$, $n\in N\left \{1,2,\cdots \right \}$, $\Gamma (x)=\int_{0}^{\infty }t^{x-1}e^{-t}dt, Re(x)>0$ is the gamma function and $f^{(n)}(\tau)=\frac{d^{n }f(\tau)}{d\tau^{n}}$.
%\par
Consider the fractional-order nonlinear system:
\begin{equation}\label{eq1}
\frac{d^{\alpha }x}{dt^{\alpha }}=g(x,\mu ),
\end{equation}
where $\alpha \in \left ( 1,2 \right )$ is the fractional order, $x=(x_1,x_2,\cdots,x_n)^T$ is the state vector, $\mu=(\mu_1,\mu_2,\cdots,\mu_n)$, $\mu_i$ are bifurcation parameters, $g(x)=\left (g_1(x),\cdots,g_n(x)\right )^T$, $g_i(\cdot)$ are nonlinear functions.
\par
Initial conditions for system (\ref{eq1}) are $x_i(0)=x_{i0}, {x}'_i(0)={x}'_{i0}$.
\par
Suppose that $x^{\ast }=(x_1^{\ast },x_2^{\ast },\cdots,x_n^{\ast })^T$ is an equilibrium point of system (\ref{eq1}). We denote the Jacobian matrix of system (\ref{eq1}) at $x^{\ast }$ by $J(\mu)$. Let the characteristic polynomial of $J(\mu)$ be
\begin{equation}\label{eq4}
f(\lambda;\alpha,\mu)=\lambda^n+a_1\lambda^{n-1}+\cdots+a_n,
\end{equation}
where $a_i=a_i(\alpha,\mu),i=1,2,\cdots,n$.
\par
From the locally asymptotical stability theorem of fractional-order nonlinear systems \cite{ahmed2007equilibrium}, for $\alpha \in \left ( 1,2 \right )$,
%system (\ref{eq1}) is stable if all roots of $f(\lambda;\alpha,\mu)$ satisfy $\left | arg\left ( \lambda \right ) \right |>\frac{\alpha \pi }{2}$.
%\par
we denote $\Omega =\left \{ z\in \mathbb{C}\mid \left | arg\left ( z \right ) \right | >\frac{\alpha \pi }{2}\right \}$, $\Sigma =\left \{ z\in \mathbb{C}\mid \left | arg\left ( z \right ) \right | <\frac{\alpha \pi }{2}\right \}$ and $\Gamma =\left \{ z\in \mathbb{C}\mid \left | arg\left ( z \right ) \right | =\frac{\alpha \pi }{2}\right \}$
%\begin{equation}
%\begin{matrix}
%\Omega =\left \{ z\in \mathbb{C}\mid \left | arg\left ( z \right ) \right | >\frac{\alpha \pi }{2}\right \},\\
%\Sigma =\left \{ z\in \mathbb{C}\mid \left | arg\left ( z \right ) \right | <\frac{\alpha \pi }{2}\right \},\\
%\Gamma =\left \{ z\in \mathbb{C}\mid \left | arg\left ( z \right ) \right | =\frac{\alpha \pi }{2}\right \},
%\end{matrix}
%\end{equation}
the stable region, the unstable region and the critical line of the complex plane, respectively.
\par
According to the fractional-order Hopf bifurcation conditions \cite{deshpande2017hopf,ma2016complexity}, system (\ref{eq1}) occurs Hopf bifurcation if $f(\lambda;\alpha,\mu)$ has a pair of conjugate complex roots, which cross the critical line $\Gamma$ at the critical roots from the stable region $\Omega$ to the unstable region $\Sigma$, as the parameter $\mu$ changes, and the other roots are all in $\Omega$.

\section{Main Results}
%Set
%\begin{small}
%\begin{equation}
%f\left ( r\cdot e^{i\frac{\alpha \pi }{2}};\alpha,\mu \right )=\sum_{j=0}^{n}a_{j}\cdot cos\left ( \frac{(n-j)\cdot \alpha \pi }{2} \right )\cdot r^{n-j}+i\cdot \left ( \sum_{j=0}^{n}a_{j}\cdot sin\left ( \frac{(n-j)\cdot \alpha \pi }{2}\right )\cdot r^{n-j} \right ),
%\end{equation}
%\end{small}
%where $a_0=1$, $a_i=a_i(\alpha,\mu),i=1,2,\cdots,n$.
%\par From the generalized Routh-Hurwitz criterion,
%The coordinate system $xy$ counterclockwise turns through angle $\theta =\frac{\left (\alpha -1  \right )\pi }{2}$ as the coordinate system ${x}'{y}'$. In the new coordinate system ${x}'{y}'$,
$f(\lambda;\alpha,\mu)$ in (\ref{eq4}) can be expressed as $g(r)=f\left (r\cdot  e^{i\frac{\alpha -1}{2}\pi };\alpha,\mu \right )$ through counterclockwise turning the coordinate system with angle $\theta =\frac{\left (\alpha -1  \right )\pi }{2}$. A generalized Routh-Hurwitz matrix is constructed from $g(ir)=f\left ( r\cdot  e^{i\frac{\alpha\pi}{2} };\alpha,\mu \right )$. Suppose that the imaginary part $f_1(r;\alpha,\mu)$ and the real part $f_2(r;\alpha,\mu)$ of $g(ir)=f\left ( r\cdot  e^{i\frac{\alpha\pi}{2} };\alpha,\mu \right )$ are expressed as
\begin{equation}\label{eq11}
\left\{\begin{matrix}
f_1(r;\alpha,\mu)=\overline{a}_0r^n+\overline{a}_1r^{n-1}+\cdots+\overline{a}_{n-1}r+\overline{a}_n\\
f_2(r;\alpha,\mu)=\overline{b}_0r^n+\overline{b}_1r^{n-1}+\cdots+\overline{a}_{n-1}r+\overline{b}_n,
\end{matrix}\right.
\end{equation}
where $\overline{a}_j=a_j\cdot sin\left ( \frac{(n-j)\cdot \alpha \pi }{2}\right )$, $\overline{b}_j=a_j\cdot cos\left ( \frac{(n-j)\cdot \alpha \pi }{2}\right )$, $j=0,1,\cdots,n$.
\begin{definition}
For $f(\lambda;\alpha,\mu)$ in (\ref{eq4}), the $2n\times2n$ fractional-order Routh-Hurwitz matrix $H_{\alpha}$ of $f(\lambda;\alpha,\mu)$ is defined as:
\begin{small}
\begin{equation}\label{eq5}
H_{\alpha}(\mu)=\left [ \begin{matrix}
\overline{a}_0  &\overline{a}_1  & \cdots  & 0 & \cdots  &\cdots & 0\\
\overline{b}_0  &\overline{b}_1 & \cdots  & \overline{b}_n & \cdots  &\cdots & 0\\
0 & \overline{a}_0 & \cdots &\overline{a}_{n-1} &0 &\cdots &0\\
0 & \overline{b}_0 &\cdots & \overline{b}_{n-1}&\overline{b}_{n}&\cdots &0\\
\vdots &\vdots &\vdots & \vdots &\vdots&\vdots &\vdots\\
0 &\cdots &\cdots & \overline{a}_0 &\cdots&\overline{a}_{n-1}&0\\
0 &\cdots &\cdots & \overline{b}_0 &\cdots&\overline{b}_{n-1} &\overline{b}_n\\
\end{matrix} \right ].
\end{equation}
\end{small}
\end{definition}
%The coordinate system $xy$ counterclockwise turns through angle $\theta =\frac{\left (\alpha -1  \right )\pi }{2}$ as the coordinate system ${x}'{y}'$. In the new coordinate system ${x}'{y}'$, $f(\lambda;\alpha,\mu)$ in (\ref{eq4}) can be expressed as $g(r)=f\left (r\cdot  e^{i\frac{\alpha -1}{2}\pi };\alpha,\mu \right )$, thus
%$g(ir)=f\left ( r\cdot  e^{i\frac{\alpha\pi}{2} };\alpha,\mu \right )$.
%The fractional-order Routh-Hurwitz matrix $H_{\alpha}$ in Eq. (\ref{eq5}) is constructed from the complex polynomial $g(ir)$.
%\par
Denote the $2p$-th order leading principle minor of $H_\alpha(\mu)$ by $\nabla_p(\mu)(p=1,2,\cdots,n)$.
\begin{theorem}\label{th1}
System (\ref{eq1}) occurs Hopf bifurcation if exists $\mu^*$ such that
\\
1.(eigenvalues condition) $\nabla_{n}(\mu^*)=0, \nabla_{p}(\mu^*)>0(p=1,2,\cdots,n-1), \widetilde{\nabla}(\mu^*)<0$, where $\widetilde{\nabla}(\mu)$ denotes the determinant of the $2\left (n-1  \right )\times2\left (n-1  \right )$ matrix extracted from the first $2n-2$ rows, the first $2n-3$ columns and the $2n-1$th columns of $H_{\alpha}(\mu)$, that is
\begin{small}
\begin{equation}
\widetilde{\nabla}(\mu)=\left | \begin{matrix}
\overline{a}_0  &\overline{a}_1  & \cdots  & 0 & \cdots  &\cdots &0\\
\overline{b}_0  &\overline{b}_1  & \cdots  & \overline{b}_n & \cdots  &\cdots &0\\
\vdots &\vdots &\vdots & \vdots &\vdots&\vdots &\vdots\\
0 &\cdots & \overline{a}_0&\cdots &\cdots & \overline{a}_{n-1} &0  \\
0 &\cdots & \overline{b}_0 &\cdots &\cdots & \overline{b}_{n-1}&0  \\
0 &\cdots &\cdots & \overline{a}_0  &\cdots& \overline{a}_{n-2}&0 \\
0 &\cdots &\cdots & \overline{b}_0  &\cdots& \overline{b}_{n-2}&\overline{b}_n \\
\end{matrix} \right |.
\end{equation}
\end{small}
\\
2.(transversality condition)
%$\exists \,\overline{\mu}\in \delta \left ( \mu^* \right )$, s.t.
%\begin{center}
%$\nabla_p\left ( \overline{\mu} \right )>0\left ( p=1,2,\cdots,n-1 \right ),\nabla_n\left ( \overline{\mu} \right )<0$,
%\end{center}
%where $\delta \left ( \mu^* \right )$ is a small enough neighborhood of $\mu^*$.
$\nabla_n(\mu)$ is indefinite in any neighborhood $\delta(\mu^*)$ of $\mu^*$, that is, $\forall \delta (\mu^*)$, $\exists \,\mu_1, \mu_2  \in \delta(\mu^*)$, s.t. $\nabla(\mu_1)>0,\nabla(\mu_2)<0$.
%$\mu$ crosses CS from $\mu_1$ to $\mu_2$ via $\mu^*$, where $CS=\left \{\mu: \nabla_n(\mu)=0, \nabla_p(\mu)(p=1,2,\cdots,n-1)>0 \right \}$.
%$\mu$ crosses $\mu^*$ along $\vec{\tau}$, where
%\begin{center}
%$\vec{\tau} \in \left \{ \vec{\tau}:\left |\vec{\tau}  \right |=1,\frac{\partial \nabla_{n}}{\partial \mu}\cdot \vec{\tau}\mid_{\mu =\mu ^*}\neq 0  \right \}$.
%\end{center}
\end{theorem}
\begin{theorem}\label{th2}
Denoting the pair of conjugate complex roots on the critical line $\Gamma$ by $\lambda(\mu^*),\overline{\lambda}(\mu^*)$, we have $\lambda(\mu^*),\overline{\lambda}(\mu^*)=-\frac{\widetilde{\nabla}(\mu^*)}{\nabla_{n-1}(\mu^*)}\cdot e^{\pm i\frac{\alpha \pi }{2}}$.
\end{theorem}
Set
\begin{equation}
BS=\left \{ \mu\mid \nabla_n(\mu)=0,\nabla_p(\mu)>0,p=1,2,\cdots,n-1 \right \}\\
\end{equation}
called the bifurcation hyper-surface surface of system (\ref{eq1}).
\par
We denote Hessian matrix $\left (\frac{\partial^2\nabla_n(\mu) }{\partial \mu_i\partial \mu_j}\right)$ of $\nabla_n(\mu)$ at $\mu=\mu^*$ by $H\left (\nabla_n(\mu^*)  \right )$.
%\begin{theorem}
%If $f(\lambda;\alpha,\mu)$ exists a pair of conjugate complex roots are "crossing" the critical line $\Gamma$ into the unstable region $\Sigma$ at $\mu=\mu^*$, and the other roots are all in the stable region $\Omega$, then
%$\exists \, \overline{\mu}\in \delta \left ( \mu^* \right )$, s.t.
%\begin{equation}
%\nabla_p\left ( \overline{\mu} \right )>0\left ( p=1,2,\cdots,n-1 \right ),\nabla_n\left ( \overline{\mu} \right )<0,
%\end{equation}
%where $\delta \left ( \mu^* \right )$ is a small enough neighborhood of $\mu^*$.
%\end{theorem}
\begin{theorem}\label{th3}
(1). If $\mu^* \in BS$ and $\frac{\partial \nabla_n(\mu)}{\partial \mu}\Big|_{\mu=\mu^*}\neq 0$, then $\nabla_n(\mu)$ is indefinite in any neighborhood $\delta(\mu^*)$ of $\mu^*$.
\\
(2). If $\mu^* \in BS$, $\frac{\partial \nabla_n(\mu))}{\partial \mu}\Big|_{\mu=\mu^*}=0$, and $H\left (\nabla_n(\mu^*)  \right )$ is indefinite, then $\nabla_n(\mu)$ is indefinite in any neighborhood ${\delta}(\mu^*)$ of $\mu^*$.
\\
(3). If $\mu^* \in BS$, $\frac{\partial \nabla_n(\mu))}{\partial \mu}\Big|_{\mu=\mu^*}=0$, and $H\left (\nabla_n(\mu^*)  \right )$ is positive (or negative) definite, then $\nabla_n(\mu)$ is positive (or negative) definite in some neighborhood $\delta(\mu^*)$ of $\mu^*$.
%$\nabla_n(\mu^*)>0$, $\nabla_p(\mu^*)>0(p=1,2,\cdots,n-1)$ and $\frac{\partial \nabla_n}{\partial \mu}\neq 0$, then $\exists \, \overline{\mu}\in \delta \left ( \mu^* \right )$, s.t.
%\begin{equation}
%\nabla_p\left ( \overline{\mu} \right )>0\left ( p=1,2,\cdots,n-1 \right ),\nabla_n\left ( \overline{\mu} \right )<0,
%\end{equation}
%where $\delta \left ( \mu^* \right )$ is a small enough neighborhood of $\mu^*$.
\end{theorem}

%\begin{theorem}
%$\nabla_n(\mu)$ is homogeneous form of degree $m$ with respect to $\sigma$ and $1-\sigma$,
%\begin{equation}
%\nabla_n(\mu)=\sum_{i=0}^{m}c_i\sigma ^{i}\left ( 1-\sigma  \right )^{m-i}=0
%\end{equation}
%where $\sigma =sin^{2}\left ( \frac{\alpha \pi }{2} \right )$, $m=\frac{n\left ( n-1 \right )}{2}$, $c_i$ are polynomials in $a_1,a_2,\cdots,a_n$.
%\end{theorem}
\section{Proofs of Theorems}
%The coordinate system $xy$ counterclockwise turns through angle $\theta =\frac{\left (\alpha -1  \right )\pi }{2}$ as the coordinate system ${x}'{y}'$.
%%(as shown in Fig. \ref{fig1}).
%%\begin{figure}[h]
%%\begin{center}
%%\includegraphics[width=2.5in]{2.jpg}
%%\end{center}
%%\caption{Rotate the coordinate system}
%%\label{fig1}
%%\end{figure}
%In the new coordinate system ${x}'{y}'$, $f(\lambda;\alpha,\mu)$ in (\ref{eq4}) can be expressed as $g(r)=f\left (r\cdot  e^{i\frac{\alpha -1}{2}\pi };\alpha,\mu \right )$,
%%\begin{equation}
%%g(r)=f\left (r\cdot  e^{i\frac{\alpha -1}{2}\pi };\alpha,\mu \right ),
%%\end{equation}
%thus
%\begin{equation}
%g(ir)=f\left ( r\cdot  e^{i\frac{\alpha\pi}{2} };\alpha,\mu \right ).
%\end{equation}
%Since $f(\lambda;\alpha,\mu)$ is a real polynomial whose roots are symmetrical about the real axis in the coordinate system $xy$, when system (\ref{eq1}) occurs Hopf bifurcation, one nonzero root of $f(\lambda;\alpha,\mu)$ is on the $y'$-axis, and the others are in the left half plane of the new coordinate system ${x}'{y}'$.
%$f(\lambda;\alpha,\mu)$ is stable if all roots in the left half plane of the coordinate system ${x}'{y}'$, that is $\nabla_{p}(p=1,2,\cdots,n)>0$, where $\nabla_{p}$ is the $2p$-th order leading principle minor of $H_{\alpha}$.
%\par
%In this paper, all notations are as above. We discuss Hopf bifurcation of fractional-order systems through the fractional-order Routh-Hurwitz matrix.
The original coordinate system and the coordinate system which counterclockwise turned angel $\theta =\frac{\left (\alpha -1  \right )\pi }{2}$ denote by $xy$ and ${x}'{y}'$, respectively. Since $f(\lambda;\alpha,\mu)$ is a real polynomial whose roots are symmetrical about the $x$-axis, when system (\ref{eq1}) occurs Hopf bifurcation, one nonzero root of $f(\lambda;\alpha,\mu)$ is on the $y'$-axis, and the others are in the left half plane of the coordinate system ${x}'{y}'$.
%Suppose that the imaginary part $f_1(r;\alpha,\mu)$ and the real part $f_2(r;\alpha,\mu)$ of $g(ir)=f\left ( r\cdot  e^{i\frac{\alpha\pi}{2} };\alpha,\mu \right )$ are expressed as
%\begin{equation}\label{eq11}
%\left\{\begin{matrix}
%f_1(r;\alpha,\mu)=\overline{a_0}r^n+\overline{a_1}r^{n-1}+\cdots+\overline{a_{n-1}}r+\overline{a_n}\\
%f_2(r;\alpha,\mu)=\overline{b_0}r^n+\overline{b_1}r^{n-1}+\cdots+\overline{a_{n-1}}r+\overline{b_n},
%\end{matrix}\right.
%\end{equation}
%where $\overline{a_j}=a_j\cdot sin\left ( \frac{(n-j)\cdot \alpha \pi }{2}\right )$, $\overline{b_j}=a_j\cdot cos\left ( \frac{(n-j)\cdot \alpha \pi }{2}\right )$, $j=0,1,\cdots,n$.
\par
\begin{lemma}\label{le1}
$f(\lambda;\alpha,\mu^*)$ has a pair of nonzero conjugate complex roots on the critical line $\Gamma$, and the other roots are in the stable region $\Omega$ if and only if
%$f_1(r;\alpha,\mu)$ and $f_2(r;\alpha,\mu)$ have a common root and the other roots are in the left half plane of the coordinate system ${x}'{y}'$, that is
$\nabla_n(\mu^*)=0$, $\nabla_p(\mu^*)>0(p=1,2,\cdots,n-1),\widetilde{\nabla}(\mu^*)<0$.
\end{lemma}
\begin{proof}
From the generalized Routh-Hurwitz criterion \cite{gantmakher2000theory}, if $f(\lambda;\alpha,\mu^*)$ has one root that is on the $y'$-axis(a part of the critical line $\Gamma$), $f_1(r;\alpha,\mu^*)$ and $f_2(r;\alpha,\mu^*)$ have one common real root. From the resultant theorem, we have $\nabla_n(\mu^*)=0,\nabla_{n-1}(\mu^*)\neq0$. Suppose that the common real root is $r_0$. Eq. (\ref{eq11}) can be rewritten as
\begin{equation}
\left\{\begin{matrix}
f_1(r;\alpha,\mu^*)=\widetilde{f}_1\cdot \left ( r-r_0\right )\\
f_2(r;\alpha,\mu^*)=\widetilde{f}_2\cdot \left ( r-r_0\right ),
\end{matrix}\right.
\end{equation}
where
\begin{equation}
\left\{\begin{matrix}
\widetilde{f}_1=\widetilde{a}_0 r^{n-1}+\widetilde{a}_1 r^{n-2}+\cdots+\widetilde{a}_{n-1}\\
\widetilde{f}_2=\widetilde{b}_0 r^{n-1}+\widetilde{b}_1 r^{n-2}+\cdots+\widetilde{b}_{n-1}.
\end{matrix}\right.
\end{equation}
Denote
\begin{small}
\begin{equation}
\widetilde{H}(\mu^*)=\left [ \begin{matrix}
\widetilde{a}_0  &\widetilde{a}_{1}&\widetilde{a}_{2}& \cdots  & 0  &\cdots &\cdots & 0\\
\widetilde{b}_0  &\widetilde{b}_{1}&\widetilde{b}_{2}& \cdots  & 0 &\cdots &\cdots & 0\\
0 & \widetilde{a}_0  &\widetilde{a}_{1}&\widetilde{a}_{2}&\cdots  & 0 & \cdots   & 0\\
0 & \widetilde{b}_0  &\widetilde{b}_{1}&\widetilde{b}_{2}& \cdots  & 0 & \cdots & 0\\
\vdots &\vdots &\vdots & \vdots &\vdots&\vdots &\vdots\\
0 &\cdots &\cdots & \widetilde{a}_0  &\widetilde{a}_{1}&\widetilde{a}_{2}& \cdots  &0\\
0 &\cdots &\cdots & \widetilde{b}_0  &\widetilde{b}_{1}&\widetilde{b}_{2}& \cdots  &0\\
\end{matrix} \right ],
\end{equation}
\end{small}
\\
the $2n\times 2n$ resultant matrix of $\widetilde{f}_1$ and $\widetilde{f}_2$, and $\widetilde{\nabla}_k(\mu^*)$ is the $2k$-th order leading principle minor of $\widetilde{H}(\mu^*)$.
Since the common roots of $\widetilde{f}_1$ and $\widetilde{f}_2$ are all in the stable region $\Omega$, from the generalized Routh-Hurwitz criterion, we have $\widetilde{\nabla}_k(\mu^*)>0(k=1,2,\cdots,n-1)$.
\par
Denote $T_i(M)$ the matrix obtained by multiplying the $(i-1)$th column of $M$ by $r_0$ and adding it to the $i$th column of $M$. We have
\begin{equation}
H_0=H_\alpha,H_1=T_1(H_0),H_2=T_2(H_1),\cdots,\widetilde{H}= T_n(H_{n-1}).
\end{equation}
Thus, $\nabla_p(\mu^*)=\widetilde{\nabla}_p(\mu^*)(p=1,2,\cdots,n-1)$, and $\nabla_p(\mu^*)(p=1,2,\cdots,n-1)>0,\nabla_n(\mu^*)=0$.
\par
From the subresultant theorem \cite{ABDELJAOUED2009588}, we have
\begin{equation}
\nabla_{n-1}(\mu^*)\cdot r_0+\widetilde{\nabla}(\mu^*)=0.
\end{equation}
Since the critical root is on the critical line $\Gamma$, which is on the left half $xy$-plane, $r_0>0$. From $\nabla_{n-1}(\mu^*)>0$, we know $\widetilde{\nabla}(\mu^*)<0$. This completes the proof of Theorem \ref{th2}.
\end{proof}
\begin{lemma}\label{le2}
If $f(\lambda;\alpha,\mu)$ has a pair of conjugate complex roots $\lambda(\mu),\overline{\lambda}(\mu)$, as $\mu$ changes, they cross the critical line $\Gamma$ at $\lambda(\mu^*),\overline{\lambda}(\mu^*)$ from the stable region $\Omega$ to the unstable region $\Sigma$, and the other roots are all in $\Omega$, then $\nabla_n(\mu)$ is indefinite in any neighborhood $\delta (\mu^*)$ of $\mu^*$, that is, $\forall \delta (\mu^*)$, $\exists\, \mu_1,\mu_2\in\delta (\mu^*)$, s.t. $\nabla_n(\mu_1)>0,\nabla_n(\mu_2)<0$.

%$\nabla_n(\mu^*)>0$, $\nabla_p(\mu^*)>0(p=1,2,\cdots,n-1)$, then
%$\exists \, \overline{\mu}\in \delta \left ( \mu^* \right )$, s.t.
%\begin{equation}
%\nabla_p\left ( \overline{\mu} \right )>0\left ( p=1,2,\cdots,n-1 \right ),\nabla_n\left ( \overline{\mu} \right )<0,
%\end{equation}
%where $\delta \left ( \mu^* \right )$ is a small enough neighborhood of $\mu^*$.
\end{lemma}
\begin{proof}
Since the other roots of $f(\lambda;\alpha,\mu)$ are all in the stable region $\Omega$, from the proof of Lemma \ref{le1}, we have $\nabla_p(\mu^*)>0(p=1,2,\cdots,n-1)$. $\nabla_p(\mu)$ is the continuous function of $\mu$, so exists a small enough neighborhood $\delta(\mu^*)$ of $\mu^*$, such that $\forall \mu_1 \in \delta(\mu^*) \cap \Omega$,  $\mu_2 \in \delta(\mu^*) \cap \Sigma$, $\nabla_p(\mu_k)>0 (k=1,2), \nabla_n(\mu_1)>0, \nabla_n(\mu_2)<0, p=1,2,...,n-1$.
\par
Lemma \ref{le1} and Lemma \ref{le2} together give the proof of Theorem \ref{th1}.
\end{proof}
Since
$\nabla_n(\mu)=\nabla_n(\mu^*)+\frac{\partial \nabla_n(\mu)}{\partial \mu}\bigg|_{\mu=\mu^*}(\mu-\mu^*)+\frac{1}{2}(\mu-\mu^*)^T H(\nabla_n(\mu^*)) (\mu-\mu^*)+o( ||\mu-\mu^*||^2),
$
Theorem \ref{th3} is obvious.
\section{Illustrative Example}
To demonstrate the effectiveness of our explicit method in the multi-parameter situation, we use the same system analyzed by the online method \cite{YANG2022111714}. Consider the fractional-order neural network system with multiple parameter as follows:
\begin{equation}\label{eq17}
\left\{\begin{matrix}
\frac{d^{\alpha }x_{1}(t)}{dt^{\alpha }}=-\mu _{1}x_{1}(t)+k_{11}f_{11}(x_{1}(t))+k_{12}f_{12}(x_{2}(t))+k_{13}f_{13}(x_{3}(t)),\\ \frac{d^{\alpha }x_{2}(t)}{dt^{\alpha }}=-\mu _{1}x_{2}(t)+k_{21}f_{21}(x_{1}(t))+k_{22}f_{22}(x_{2}(t))+k_{23}f_{23}(x_{3}(t)),
\\ \frac{d^{\alpha }x_{3}(t)}{dt^{\alpha }}=-\mu _{2}x_{3}(t)+k_{31}f_{31}(x_{1}(t))+k_{32}f_{32}(x_{2}(t))+k_{33}f_{33}(x_{3}(t)),
\end{matrix}\right.
\end{equation}
where $\mu _{1},\mu_{2},k_{ij}(i,j=1,2,3)$ are bifurcation parameters, $\alpha =1.1$, $f_{ij}(x_{j})=tanh(x_{j}(t))$, the initial conditions are $x_j(0)=x_{j0}, {x}'_j(0)={x}'_{j0}$. The equilibrium point of system (\ref{eq17}) is $(0,0,0)$.
\par
%The Jacobian matrix $J(\mu)$ of system (\ref{eq17}) is:
%\begin{equation}
%J(\mu)=\begin{bmatrix}
%-\mu _{1}+k_{11} & k_{12} & k_{13}\\
%k_{21} & -\mu _{1}+k_{22} & k_{23}\\
%k_{31} & k_{32} & -\mu _{2}+k_{33}.
%\end{bmatrix}
%\end{equation}
The characteristic polynomial of the Jacobian matrix $J(\mu)$ is:
\begin{equation}
f(\lambda;\alpha,\mu) = \lambda ^{3}+a_1(\alpha,\mu)\lambda ^{2}+a_2\lambda(\alpha,\mu)+a_3(\alpha,\mu)=0,
\end{equation}
where
\begin{equation}
\begin{aligned}
a_1(\alpha,\mu)=&\mu_2-k_{33}+2\mu_1-k_{22}-k_{11},\\
a_2(\alpha,\mu)=&k_{11}k_{22}+k_{11}k_{33}-k_{11}\mu_1-k{11}\mu_2-k_{12}k_{21}-k_{13}k_{31}\\&+k_{22}k_{33}-k_{22}\mu_1-k_{22}\mu_2-k_{23}k_{32}-2k_{33}\mu_1+\mu_1^2+2\mu_1\mu_2,\\
a_3(\alpha,\mu)=&-k_{11}k_{22}k_{33}+k_{11}k_{22}\mu_2+k_{11}k_{23}k_{32}+k_{11}k_{33}\mu_1-k_{11}\mu_1\mu_2\\&+k_{12}k_{21}k_{33}-k_{12}k_{21}\mu_2-k_{12}k_{23}k_{31}-k_{13}k_{21}k_{32}+k_{13}k_{22}k_{31}\\&-k_{13}k_{31}\mu_1+k_{22}k_{33}\mu_1-k_{22}\mu_1\mu_2-k_{23}k_{32}\mu_1-k_{33}\mu_1^2+\mu_1^2\mu_2.
\end{aligned}
\end{equation}
%From
%\begin{equation}
%\begin{aligned}
%f(r\cdot e^{i\frac{11\pi }{20}}) &= \sum_{j=0}^{3}a_j\cdot cos\left ( \frac{\left ( 3-j \right )\cdot 11 \pi }{20} \right )\cdot r^{3-j}+i\cdot\left ( \sum_{j=0}^{3}a_j\cdot sin\left ( \frac{\left ( 3-j \right )\cdot 11 \pi }{20} \right )\cdot r^{3-j} \right ) \\
% &= f_1(r)+i\cdot f_2(r)=0,\\
%\end{aligned}
%\end{equation}
%where $a_j=a_j(\alpha,\mu)$,
From the definition in (\ref{eq5}) and Theorem \ref{th1}, we have the $H_{\alpha}(\mu)$ and $\widetilde{\nabla}(\mu)$ of system (\ref{eq17}).
%\begin{equation}
%H_{\alpha}(\mu)=\begin{bmatrix}
%sin\left ( \frac{33\pi }{20} \right ) & a_1sin\left ( \frac{22\pi }{20} \right ) & a_2sin\left ( \frac{11\pi }{20} \right ) & 0 & 0 & 0\\
%cos\left ( \frac{33\pi }{20} \right ) & a_1cos\left ( \frac{22\pi }{20} \right ) & a_2cos\left ( \frac{11\pi }{20} \right ) & a_3 & 0 & 0\\
%0 & sin\left ( \frac{33\pi }{20} \right ) & a_1sin\left ( \frac{22\pi }{20} \right ) & a_2sin\left ( \frac{11\pi }{20} \right ) & 0 & 0 \\
%0 & cos\left ( \frac{33\pi }{20} \right ) & a_1cos\left ( \frac{22\pi }{20} \right ) & a_2cos\left ( \frac{11\pi }{20} \right ) & a_3 & 0\\
%0 & 0 &sin\left ( \frac{33\pi }{20} \right ) & a_1sin\left ( \frac{22\pi }{20} \right ) & a_2sin\left ( \frac{11\pi }{20} \right ) & 0\\
%0 & 0 & cos\left ( \frac{33\pi }{20} \right ) & a_1cos\left ( \frac{22\pi }{20} \right ) & a_2cos\left ( \frac{11\pi }{20} \right ) & a_3
%\end{bmatrix},
%\end{equation}
%and
%\begin{equation}
%\widetilde{\nabla}(\mu)=\begin{vmatrix}
%sin\left ( \frac{33\pi }{20} \right ) & a_1sin\left ( \frac{22\pi }{20} \right ) & a_2sin\left ( \frac{11\pi }{20} \right )  & 0 \\
%cos\left ( \frac{33\pi }{20} \right ) & a_1cos\left ( \frac{22\pi }{20} \right ) & a_2cos\left ( \frac{11\pi }{20} \right )  & 0 \\
%0 & sin\left ( \frac{33\pi }{20} \right ) & a_1sin\left ( \frac{22\pi }{20} \right ) &  0  \\
%0 & cos\left ( \frac{33\pi }{20} \right ) & a_1cos\left ( \frac{22\pi }{20} \right ) & a_3
%\end{vmatrix}.
%\end{equation}
Then
%\begin{equation}
%\begin{aligned}
%\nabla_1(\mu) &=a_1sin\left ( \frac{9\pi }{20} \right ), \\
%\nabla_2(\mu) &=\left ( \frac{a_1^2a_2}{2}-\frac{a_3a_1}{2} \right )cos\left ( \frac{\pi }{10} \right )+\left ( \frac{a_2^2}{2}-\frac{a_1a_3}{2} \right )cos\left ( \frac{\pi }{5} \right )+\frac{a_1^2a_2}{2}-\frac{a_2^2}{2},\\
%\nabla_3(\mu) &=\frac{a_3}{8}(2a_3^2sin\left ( \frac{\pi }{20} \right )+2a_1a_2a_3sin\left ( \frac{3\pi }{20} \right )+( 2a_1^3a_3+2a_1^2a_2^2-12a_1a_2a_3\\&+2a_2^3+6a_3^2 )sin\left ( \frac{7\pi }{20} \right )+ ( 6a_1^2a_2^2-4a_1^3a_3-2a_2a_1a_3-4a_2^3 )sin\left ( \frac{9\pi }{20} \right )\\&+\sqrt{2}a_1^3a_3-2\sqrt{2}a_1a_2a_3+\sqrt{2}a_2^3),\\
%\widetilde{\nabla}(\mu) &=-\frac{1}{4}a_3\left ( 2a_1^2cos\left ( \frac{7\pi }{20} \right )+\left ( 2a_1^2-2a_2 \right )cos\left ( \frac{9\pi }{20} \right )+\sqrt{2}a_2 \right ).
%\end{aligned}
%\end{equation}
\begin{equation}\nonumber
\begin{aligned}
\nabla_1(\mu) &=a_1sin\left ( \frac{9\pi }{20} \right ), \\
\nabla_2(\mu) &=\left ( \frac{a_1^2a_2}{2}-\frac{a_3a_1}{2} \right )cos\left ( \frac{\pi }{10} \right )+\left ( \frac{a_2^2}{2}-\frac{a_1a_3}{2} \right )cos\left ( \frac{\pi }{5} \right )+\frac{a_1^2a_2}{2}-\frac{a_2^2}{2},\\
\end{aligned}
\end{equation}
\begin{equation}
\begin{aligned}
\nabla_3(\mu) &=\frac{a_3}{8}(2a_3^2sin\left ( \frac{\pi }{20} \right )+2a_1a_2a_3sin\left ( \frac{3\pi }{20} \right )+( 2a_1^3a_3+2a_1^2a_2^2-12a_1a_2a_3\\&+2a_2^3+6a_3^2 )sin\left ( \frac{7\pi }{20} \right )+ ( 6a_1^2a_2^2-4a_1^3a_3-2a_2a_1a_3-4a_2^3 )sin\left ( \frac{9\pi }{20} \right )\\&+\sqrt{2}a_1^3a_3-2\sqrt{2}a_1a_2a_3+\sqrt{2}a_2^3),\\
\widetilde{\nabla}(\mu) &=-\frac{1}{4}a_3\left ( 2a_1^2cos\left ( \frac{7\pi }{20} \right )+\left ( 2a_1^2-2a_2 \right )cos\left ( \frac{9\pi }{20} \right )+\sqrt{2}a_2 \right ).
\end{aligned}
\end{equation}
Based on Theorem \ref{th1}, the bifurcation hyper-surface $BS$ of system (\ref{eq17}) is determined by
\begin{figure}[h]
\begin{center}
\includegraphics[width=2.5in]{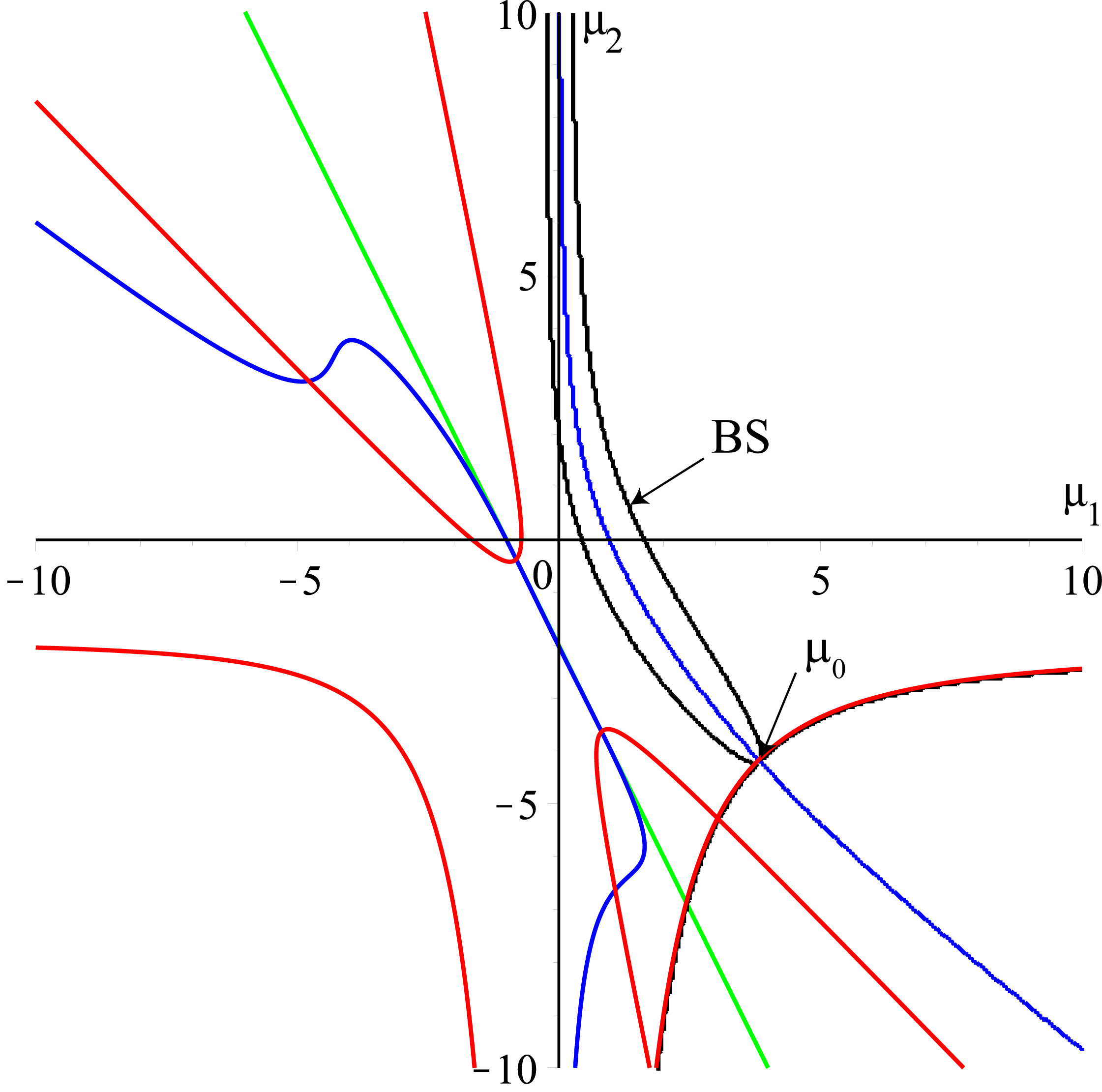}
\end{center}
\caption{The bifurcation hyper-surface in the $(\mu_1,\mu_2)$-plane for fixed $k_{ij}$}
\label{fig2}
\end{figure}
\begin{equation}\label{eq24}
\nabla_1(\mu^*)>0,\nabla_2(\mu^*)>0,\nabla_3(\mu^*)=0,\widetilde{\nabla}(\mu^*)<0.
\end{equation}
To make the results intuitive, we assume that $\mu_{1}$ and $\mu_2$ are bifurcation parameters, other parameters are fixed as $k_{11}=k_{12}=k_{13}=k_{23}=2,k_{21}=k_{22}=k_{31}=k_{33}=-2$, $k_{32}=1$. Fig. \ref{fig2} shows the bifurcation surface of system (\ref{eq17}) in this case.
\par
From Theorem \ref{th3}, we know that there exists the only one point $\mu_0\in BS$ which satisfies
\begin{center}
$\nabla_3(\mu)=0$, $\frac{\partial \nabla_n(\mu)}{\partial \mu}\Big|_{\mu=\mu_0}=0$, and $H\left (\nabla_n(\mu)  \right )\Big|_{\mu=\mu_0}$ is negative definite,
\end{center}
so the $\mu_0\approx (3.817533638,-4.170716050)$ does not meet the transversality condition.
\par
If the number of uncertain parameters are more three, it is difficult to show $BS$ with a figure. Our explicit expressions of results in Eq. (\ref{eq24}) are still valid for the situation of multi-parameter.
%\begin{remark}
%Eq. (\ref{eq24}) are the Hopf bifurcation results of system (\ref{eq17}), Fig. \ref{fig2} is just a visual representation of our results. It is noted that when the number of bifurcation parameters is greater than three, even if all results are hard to be shown through a visual representation, our explicit method still gives complete Hopf bifurcation results.
%\end{remark}
%\begin{remark}
%If fractional-order systems are of multiple-parameter, then the corresponding Jacobian matrix has the infinite number of eigenvalues. The existing online method of analyzing eigenvalues may fail to obtain the critical values of multiple bifurcation parameters due to computational complexity.
%\end{remark}

\section{Conclusions}
Explicit Hopf bifurcation conditions of fractional-order systems with order $1<\alpha <2$ are proposed in this paper. Basing the generalized Routh-Hurwitz criterion and subresultant theorem, we obtain some matrices to determine Hopf bifurcation hyper-surface directly. Our explicit method can avoid the failure of the existing methods in the face of cases of a large number of parameters.
\section*{Acknowledgment}
This work was partially supported by the National Natural Science Foundation of China
under Grants no.12171073.
\section*{References}
\bibliographystyle{elsarticle-num-names}
\biboptions{square,numbers,sort&compress}
\bibliography{elsarticle-template}

\end{document}